\documentclass[11pt]{amsart}
\usepackage{amsmath,amssymb,latexsym,soul,cite,mathrsfs}

\usepackage{color,enumitem,graphicx}
\usepackage[colorlinks=true,urlcolor=blue,
citecolor=red,linkcolor=blue,linktocpage,pdfpagelabels,
bookmarksnumbered,bookmarksopen]{hyperref}
\usepackage[english]{babel}

\usepackage[left=2.9cm,right=2.9cm,top=2.8cm,bottom=2.8cm]{geometry}
\usepackage[hyperpageref]{backref}

\usepackage[colorinlistoftodos]{todonotes}
\makeatletter
\providecommand\@dotsep{5}
\def\listtodoname{List of Todos}
\def\listoftodos{\@starttoc{tdo}\listtodoname}
\makeatother

\numberwithin{equation}{section}
\def\dis{\displaystyle}
\def\cal{\mathcal}
\newtheorem{theorem}{Theorem}[section]
\newtheorem{lemma}[theorem]{Lemma}
\newtheorem{proposition}[theorem]{Proposition}

\providecommand{\abs}[1]{\lvert#1\rvert}
\providecommand{\norm}[1]{\lVert#1\rVert}

\DeclareMathOperator{\cat}{cat}

\title[fractional Schr\"odinger-Poisson system]
{Positive semiclassical states  for a \\ fractional Schr\"odinger-Poisson system}

\author[E. G. Murcia]{ Edwin Gonzalo Murcia}
\author[G. Siciliano]{Gaetano Siciliano}
\address[E. G. Murcia, G. Siciliano ]{\newline\indent Departamento de Matem\'atica
\newline\indent 
Instituto de Matem\'atica e Estat\'istica
\newline\indent 
 Universidade de S\~ao Paulo 
\newline\indent 
Rua do Mat\~ao 1010,  05508-090 S\~ao Paulo, SP, Brazil }
\email{\href{mailto:edwingmr@ime.usp.br}{edwingmr@ime.usp.br}}
\email{\href{mailto:sicilian@ime.usp.br}{sicilian@ime.usp.br}}
\thanks{Edwin G. Murcia is
supported by  Capes, Brazil. Gaetano Siciliano  is supported by
Fapesp and CNPq, Brazil. }
\begin{document}

\maketitle

\begin{abstract}
We consider a fractional Schr\"odinger-Poisson system in the whole space $\mathbb R^{N}$
in presence of a positive potential and depending on a small positive parameter $\varepsilon.$ We 
show that, for suitably small $\varepsilon$ (i.e. in the ``semiclassical limit'') the number of positive solutions is estimated below by the
Ljusternick-Schnirelmann category of the set of minima of the potential.

\end{abstract}



\section{Introduction}
In the last decades a great attention has been given to the following Schr\"odinger-Poisson type system
\begin{equation*}
\left\{
\begin{array}{l}
-\Delta u + V(x)u+\phi u = |u|^{p-2}u \\
-\Delta \phi = u^{2},
\end{array}
\right.
\end{equation*}
which arises
 in non relativistic Quantum Mechanics. Such a system is obtained by looking  for standing waves solutions
 in the purely electrostatic case to the  Schr\"odinger-Maxwell system. For a deduction of this system, see e.g. \cite{BF}.
 Here
 the unknowns are $u$, the modulus of the wave function, and $\phi$ which represents the electrostatic potential.
 $V$ is a given external potential and $p\geq2$ a suitable given number.

The system has been studied by many authors, both in bounded and  unbounded domains, with different assumptions on the data involved: boundary conditions, potentials, nonlinearities; many different type of solutions have been encountered
(minimal energy, sign changing, radial, nonradial...),     the behaviour
of the solutions (e.g.  concentration phenomena)  has been studied as well as 
 multiplicity results have been obtained.
It is really difficult to give a complete
list of references: the reader may see \cite{BFort} and the references therein.

However  it seems that  results relating the number of positive solutions 
with topological invariants of the ``objects'' appearing in the problem
are few in the literature.
We cite the paper \cite{Sicilia} where the system
is studied in a (smooth and) bounded domain $\Omega\subset  \mathbb R^{3}$ with $u=\phi = 0$ on $\partial\Omega$
and $V$ constant.
It is shown, by 
using 
variational methods
that, whenever $p$ is sufficiently near the critical Sobolev exponent $6$, the number of positive solutions is estimated below by the {\sl Ljusternick-Schnirelamnn category}
of the domain $\Omega$. 

\medskip

On the other hand it is known that a particular interest has the {\sl semiclassical limit} of the Schr\"odinger-Poisson system
(that is when the Plank constant $\hbar$ appearing in the system, see e.g. \cite{BF}, tends to zero) especially due to the fact
that this limit describes the transition from Quantum
to Classical Mechanics. Such a situation is studied e.g. in
\cite{R}, among many other papers.
We cite also Fang and Zhao  \cite{fangzhang} which consider
 the following doubly perturbed system
 in the whole space $\mathbb R^{3}$:
\begin{equation*}\label{fzproblem}
\left\{
\begin{array}{l}
- \varepsilon^{2}\Delta w+V(x)w+\psi w=|w|^{p-2}w \smallskip \\
-\varepsilon\Delta \psi= w^2 .
\end{array}
\right.
\end{equation*}
Here $V$  is a suitable potential,  $4<p<6$,
and $\varepsilon$ is a positive parameter proportional to $\hbar$.
In this case the authors estimate, whenever $\varepsilon$ tends to zero, 
the number of positive solutions by the 
Ljusternick-Schnirelamnn category of the set of minima of the potential $V$,
obtaining a result in the same spirit of \cite{Sicilia}.

\medskip

Recently, especially  after the formulation of the Fractional Quantum Mechanics, the derivation of the Fractional Schr\"odinger 
equation given  by N. Laskin  in  \cite{L1,L2,L3}, and the notion of fractional harmonic extension of a function studied in the 
pioneering  paper \cite{CSil}, equations involving fractional operators are receiving a great attention. Indeed pseudodifferential
operators appear in many problems in Physics and Chemistry, see e.g. \cite{MK1,MK2}; but also in
obstacle problems \cite{MS, S}, optimization and finance \cite{CT}, conformal geometry and minimal surfaces
\cite{CVald,CGonz,CRS}, etc.

\bigskip

Motivated by the previous discussion, we investigate in this paper the existence of  positive solutions for the following 
doubly singularly perturbed fractional Schr\"odinger-Poisson system
in $\mathbb R^{N}$:
\begin{equation}\label{prob1}\tag{$P_{\varepsilon}$}
\left\{
	\begin{array}{l}
		\varepsilon^{2s}\left(-\Delta \right)^{s}w+V(x)w+\psi w=f(w) \smallskip\\
		\varepsilon^{\theta}\left(-\Delta \right)^{\alpha /2} \psi= \gamma_{\alpha} w^{2},
	\end{array}
\right.
\end{equation}
where $\gamma_{\alpha} :=
 \frac{\pi ^{N/2} 2^\alpha \Gamma \left( \alpha /2\right)}{\Gamma \left( N/2-\alpha /2\right)}$
is a  constant ($\Gamma$ is the Euler function).
By a positive solution of \eqref{prob1} we mean a pair $(w,\psi)$ where $w$ is positive.
To the best of our knowledge, there are only few  recent papers dealing with a system like \eqref{prob1}: in \cite{Zhang}
the author deals with $\varepsilon=1$ proving under suitable assumptions on $f$ the existence of infinitely many 
(but possibly sign changing) solutions
by means of the Fountain Theorem. A similar system is studied in \cite{Wei} and the existence of
infinitely many  (again, possibly sign changing) solutions is obtained by means of the Symmetric Mountain Pass Theorem.

\medskip

In this paper we  assume that \medskip
\begin{enumerate}[label=(H\arabic*),ref=H\arabic*,start=1]
\item\label{H} $s\in (0,1)$, $\alpha\in (0,N)$, $\theta \in (0, \alpha)$, $N\in (2s, 2s+\alpha)$,
\end{enumerate}
\medskip
moreover the potential $V$ and the nonlinearity $f$ satisfy the assumptions listed below: \smallskip
\begin{enumerate}[label=(V\arabic*),ref=V\arabic*]
\item\label{V} $V:\mathbb R^{N}\to \mathbb R$ is a continuous function and
$$0< \min_{\mathbb R^{N}} V:=V_{0} < V_{\infty}:=\liminf_{|x|\to +\infty} V\in(V_{0},+\infty];$$
\end{enumerate}

\begin{enumerate}[label=(f\arabic*),ref=f\arabic*,start=1]
\item\label{f_{1}} $f: \mathbb R \rightarrow \mathbb R $ is a function of class $C^{1}$ and
$f(t)=0$ for $t\leq0$; \medskip
\item\label{f_{2}} $\lim_{t \rightarrow 0} {f(t)/t}=0$; \medskip 
\item\label{f_{3}}  there is $q_{0}\in  (2, 2^{*}_{s}-1)$ such that $\lim_{t\rightarrow \infty}{f(t)}/{t^{q_{0}}}=0$,
 where $2_{s}^{*}:=2N/(N-2s)$; \medskip
\item\label{f_{4}}  there is $K>4$ such that $0<K F(t):=K\int^{t}_{0}f(\tau)d\tau\leq tf(t)$ for all $ t>0$; \medskip 
\item\label{f_{5}}  the function $ t\mapsto f(t)/t^3$ is strictly increasing in $(0,+\infty).$
\end{enumerate}

\medskip

The assumptions on the nonlinearity $f$ are quite standard in order to work with variational methods, use
the Nehari manifold and the Palais-Smale condition. 
The assumption \eqref{V}  
will be fundamental in order to estimate the number of positive solutions and also to recover some compactness.

We recall, once for all, that a $C^{1}$ functional $\mathcal J$, defined on a smooth manifold $\mathfrak M$, is said to satisfy the Palais-Smale condition at level $c\in \mathbb R$
($(PS)_{c}$ for brevity) if every sequence $\{u_{n}\}\subset \mathfrak M$ such that 
\begin{equation}\label{eq:PSseq}
\mathcal J(u_{n})\to c\quad \text{  and } \quad \mathcal J'(u_{n})\to 0
\end{equation} 
has a convergent subsequence. 
A sequence $\{u_{n}\}$ satisfying \eqref{eq:PSseq} is also named a $(PS)_{c}$ sequence.

\medskip

To stay our result let us introduce 
$$
M:=\Big\{x\in\mathbb R^{N}:V(x)=V_{0}\Big\}
$$
 the set of minima of $V$. Our result is the following
\begin{theorem}\label{th:main}
Under the above assumptions \eqref{H}, \eqref{V}, \eqref{f_{1}}-\eqref{f_{5}}, there exists an $\varepsilon^{*}>0$
such that for every $\varepsilon\in (0,\varepsilon^{*}]$ problem \eqref{prob1} possesses at least $\cat M$ positive solutions.

Moreover if $\cat M>1$ and $M$ is bounded, then (for suitably small $\varepsilon$) there exist at least $\cat M+1$ positive solutions.
\end{theorem}

Hereafter, given a topological pair $(X,Y),\cat_{X}(Y)$ is the Ljusternick-Schnirelmann category of $Y$ in $X$,
and, if $X=Y$ this is just denoted with $\cat X$.

\medskip

The proof of Theorem \ref{th:main} is carried out by adapting some ideas of Benci, Cerami  and Passaseo \cite{BC1, BCP} and using the Ljusternick-Schnirelmann Theory.
We mention that these ideas and techniques have been  extensively used to attack also other type of  problems, and indeed
similar  results are obtained for other  equations and operators, like
the Schr\"odinger operator  \cite{CL1,CL2},
 the $p-$laplacian \cite{CS,Alvesgio1}, the biharmonic operator 
\cite{AFMMAS},
$p\&q-$laplacian, fractional laplacian \cite{gaetanogiovany, GMG}, magnetic laplacian \cite{AFMILAN, AFMM}
or quasilinear operators \cite{AFig,AFS, AUG}.

\medskip The plan of the paper is the following. In Section \ref{sec:VS} we recall some basic facts,
we present some preliminaries and the variational setting for the problem. Section \ref{compattezza}
is devoted to prove some compactness properties; as a byproduct we prove the existence
of a ground state solution for our problem, that is a solution having minimal energy. In Section \ref{sec:bary} we introduce the barycenter map,
we show some of its properties and prove, by means of the Ljusternick-Schnirelamnn Theory, Theorem \ref{th:main}.

%

\bigskip

\noindent{\bf Notations.} In the paper we will denote with $|\cdot |_{p}$ the usual $L^{p}$ norm
in $\mathbb R^{N}$; we denote with $B_{r}(x)$
the closed ball in $\mathbb R^{N}$ centered in $x$ with radius $r>0$, with $B^{c}_{r}(x)$ its complementary;
if $x=0$ we simply write $B_{r}$; moreover the letters $C,C_{1}, C_{2},\ldots$
will denote  generic positive constants
(whose value may change from line to line). 
Other notations will be introduced whenever we need.

\section{Preliminaries}\label{sec:VS}
\subsection{Some well known facts}
Before to introduce the variational setting of our problem, we 
recall some basic facts concerning the fractional Sobolev spaces and their embeddings.

Given $\beta\in (0,1)$, the fractional Laplacian $(-\Delta)^\beta$ is the pseudodifferential operator  which can be
 defined via the Fourier transform
$$\mathcal F((-\Delta)^{\beta}u)=|\cdot|^{2\beta}\mathcal Fu,$$
or, if $u$ has sufficient regularity,  by
\[
(-\Delta)^\beta u(z)=-\frac{C_{N,\beta}}{2}\int_{\mathbb R^{N}} \frac{u(z+y)-u(z-y)-2u(z)}{|y|^{N+2\beta}}dy, \quad z\in\mathbb R^N,
\]
where $C_{N,\beta}$ is a suitable normalization constant.

For  $s \in (0,1)$ let
\[
H^s(\mathbb R^N)=\left\{u\in L^2(\mathbb R^N):\,
(-\Delta)^{s/2}u \in L^2(\mathbb R^N)\right\}
\]
be  the Hilbert space with scalar product and (squared) norm given by
\[
(u,v)= \int_{\mathbb R^{N}} (-\Delta)^{s/2}u (-\Delta)^{s/2}v + \int_{\mathbb R^{N}} uv,
\qquad
\|u\|^2=|(-\Delta)^{s/2}u |_2^2+ |u|_2^2.
\]
It is known that $H^{s}(\mathbb  R^{N})  \hookrightarrow L^{p}(\mathbb R^{N}), p\in [2, 2^{*}_{s}]$ with $2^{*}
_{s}:=2N/(N-2s)$. Moreover the embedding of $H^{s}(\Omega)$ is compact if $\Omega\subset \mathbb R^{N}$
is bounded and $p\neq 2^{*}_{s}.$

We will   consider also the homogeneous Sobolev spaces 
$\dot{H}^{\alpha/2}(\mathbb R^{N})$ defined as the completion of $C^{\infty}_{c}(\mathbb R^{N})$ with respect to the norm
$|(-\Delta)^{\alpha/4} u|_{2}$. 
This is a Hilbert space with scalar product and (squared) norm
\[
(u,v)_{\dot{H}^{\alpha/2}}= \int_{\mathbb R^{N}} (-\Delta)^{\alpha/4}u (-\Delta)^{\alpha/4}v,
\qquad
\|u\|_{\dot{H}^{\alpha/2}}^2=|(-\Delta)^{\alpha/4}u |_2^2.
\]
It is well known that 
$\dot{H}^{\alpha /2}( \mathbb R^N ) \hookrightarrow L^{2^{*}_{\alpha /2}}( \mathbb R^N ), 2^{*}_{\alpha}=2N/(N-\alpha)$.
For more general facts about the fractional Laplacian we refer the reader to the beautiful paper \cite{DPV}.

\bigskip

We recall here another fact  that will be frequently used:
\begin{equation}\label{terminof}
\forall \varepsilon>0 \ \exists M_{\xi}>0\ :\quad  \int_{\mathbb R^N} f( u)u \leq \xi \int_{\mathbb R^N}u^2+M_\xi \int_{\mathbb R^N}\abs{u}^{q_{0}+1}, \quad \forall u\in H^{s}(\mathbb R^{N}).
\end{equation}
 This simply follows by \eqref{f_{2}} and \eqref{f_{3}}.

\bigskip

\subsection{The variational setting}

It is easily seen that,  just performing the change of variables 
$w(x):=u(x/\varepsilon ), \psi(x):=\phi(x/\varepsilon)$, problem
 \eqref{prob1} can be rewritten as
%
\begin{equation}\label{prob2}\tag{$P_{\varepsilon}^{*}$}
\left\{
	\begin{array}{l}
		\left(-\Delta \right)^{s}u+V(\varepsilon x)u+\phi (x)u=f(u)\medskip \\ 
		(-\Delta )^{\alpha /2} \phi=\varepsilon^{\alpha -\theta} \gamma_{\alpha}u^{2}, \\ 
	\end{array}
\right.
\end{equation}
to which we will refer from now on.

A usual ``reduction'' argument can be used to deal with a single equation involving just $u$.
Indeed for every $u\in H^{s}(\mathbb R^{N})$ the second equation in \eqref{prob2} is uniquely solved.
Actually, for future reference, we will prove a slightly more general fact.

Let us fix two functions $u,w\in H^{s}(\mathbb R^{N})$ and consider the problem
\begin{equation}\label{prob3}\tag{$Q_{\varepsilon}$}
\left\{
    \begin{array}{l}
        	\left(-\Delta \right)^{\alpha /2} \phi=\varepsilon ^{\alpha-\theta} \gamma_{\alpha} uw, \medskip \\
        	\phi \in \dot{H}^{\alpha /2} ( \mathbb{R}^{N} )
    \end{array}
\right.
\end{equation}
whose weak solution is a function $\tilde \phi\in \dot H^{\alpha/2}(\mathbb R^{N})$ such that
$$\forall v\in \dot H^{\alpha/2}(\mathbb R^{N}): \quad \int_{\mathbb R^{N} } (-\Delta)^{\alpha/4} \tilde \phi (-\Delta)^{\alpha/4} v = \varepsilon^{\alpha-\theta}\gamma_{\alpha}\int_{\mathbb R^{N}} uw v.$$
For every $v\in \dot{H}^{\alpha /2} ( \mathbb{R}^{N})$, by the H\"older inequality and the continuous embeddings, we have
\[
\Big|\int_{\mathbb R^{N}} uwv \Big|\leq   \abs{u}_{\frac{4N}{N+\alpha}}|w|_{\frac{4N}{N+\alpha}} \abs{v}_{2^{*}_{\alpha /2}} 
\leq  C\|u\|\|w\| \norm{v}_{\dot{H}^{\alpha /2}}
\]
deducing that  the map
$$
T_{u,w}:v\in \dot{H}^{\alpha /2}( \mathbb R^N ) \longmapsto  \int_{\mathbb R^{N}} uw v \in \mathbb R
$$
 is linear and continuous: then there exists a unique solution
$ \phi_{\varepsilon,u,w} \in \dot{H}^{\alpha /2}( \mathbb R^N )$  to \eqref{prob3}.
Moreover this  solution has the representation by means of the Riesz kernel $\mathcal K_{\alpha}(x)=\gamma_{\alpha}^{-1}|x|^{\alpha-N}$, hence
$$\phi_{\varepsilon,u,w}
= \varepsilon ^{\alpha-\theta} \frac{1}{|\cdot |^{N-\alpha}}\star (u w).$$

Furthermore
\begin{equation}\label{eq:10}
\norm{\phi_{\varepsilon, u,w}}_{\dot{H}^{\alpha /2}}=\varepsilon^{\alpha-\theta}\|T_{u,w}\|_{\mathcal L (\dot H^{\alpha/2};\mathbb R)}
\leq \varepsilon^{\alpha-\theta}C \norm{u}\|w\|
\end{equation}
and then, for $\zeta, \eta\in H^{s}(\mathbb R^{N})$
\begin{equation}\label{terminoconvolucion}
\int_{\mathbb R^N} \phi_{\varepsilon, u,w}\zeta\eta\leq \abs{\phi_{\varepsilon ,u,w}}_{2^{*}_{\alpha /2}}\abs{\zeta}_{\frac{4N}{N+\alpha}}|\eta|_{\frac{4N}{N+\alpha}}
\leq  \varepsilon^{\alpha -\theta}C_{\textrm e}\norm{u}\|w\|\|\zeta\|\|\eta\|
\end{equation}
where $C_{\textrm e}$ is a suitable embedding constant. Altough its value is not important, we will refer to this constant later on.
%

A particular case of the previous situation is   when $u=w$. In this case we simplify the notation and write 
\begin{itemize}
\item $T_{u}(v):= T_{u,u}(v) = \int_{\mathbb R^{N}} u^{2}v$, and \medskip
\item $\phi_{\varepsilon,u}$ for  the unique solution of the second equation in \eqref{prob2} for fixed $u\in H^{s}(\mathbb R^{N})$. Then
\begin{equation*}\label{eq:questa}
\|\phi_{\varepsilon,u}\|_{\dot{H}^{\alpha/2}}\leq \varepsilon^{\alpha-\theta} C \|u\|^{2}
\end{equation*}
and  the map
$$u \in H^{s}(\mathbb R^{N}) \longmapsto \phi_{\varepsilon, u} \in \dot{H}^{\alpha/2}(\mathbb R^{N})$$
is bounded.
\end{itemize}
Observe also that 
 \begin{equation}\label{eq:convergfacil}
u_{n}^{2}  \to u^{2} \ \text{in }  \ L^{\frac{2N}{N+\alpha}}(\mathbb R^{N})
\Longrightarrow  T_{u_{n}} \longrightarrow T_{u} \  \text{ as operators } \Longrightarrow \phi_{\varepsilon,u_{n}} \longrightarrow \phi_{\varepsilon, u}
\text{ in } \dot{H}^{\alpha/2}(\mathbb R^{N}).
\end{equation}
 
For convenience let us define the map (well defined by \eqref{terminoconvolucion})
\begin{equation*}\label{eq:N}
A: u\in H^{s}(\mathbb R^{N}) \longmapsto  \int_{\mathbb R^N} \phi_{\varepsilon, u}u^2  \in \mathbb R.
\end{equation*}
Then
\begin{equation}\label{eq:esaseusa}
|A(u)| \leq \varepsilon^{\alpha-\theta} C_{\textrm{e}}\|u\|^{4}
\end{equation}
(where $C_{\textrm e}$ is the same constant in \eqref{terminoconvolucion}).
Some relevant properties of $\phi_{\varepsilon, u}$ and $A$ are listed below.
Although these properties are known to be true, we are not able to find them explicitely in the literature; so we 
prefer to give a proof here.

\begin{lemma}\label{propiedades}
The following propositions hold.
\begin{itemize}
\item[$(i)$] For every  $u\in H^{s}(\mathbb R^{N}): \phi_{\varepsilon, u}\geq 0$; \smallskip
\item[$(ii)$] for every  $u\in H^{s}(\mathbb R^{N}), t\in \mathbb R: \phi_{\varepsilon, tu}=t^2 \phi_{\varepsilon, u}$; \smallskip
\item[$(iii)$] if $u_n \rightharpoonup u$ in $H^s( \mathbb R^N )$ then $\phi_{\varepsilon, u_n} \rightharpoonup \phi_{\varepsilon, u}$ in $\dot{H}^{\alpha/2}(\mathbb R^{N})$; \smallskip
\item[$(iv)$] $A$ is of class $C^{2}$ and for every $u,v,w\in H^{s}(\mathbb R^{N})$ 
$$A'(u)[v] = 4 \int_{\mathbb R^N} \phi_{\varepsilon, u}uv,\quad A''(u)[v,w] = 4\int_{\mathbb R^N} \phi_{\varepsilon, u}vw+ 
8\int_{\mathbb R^N} \phi_{\varepsilon, u, w} uv,$$
\smallskip
\item[$(v)$] if $u_n \rightarrow u$ in $L^r ( \mathbb R^N)$, with $2\leq r< 2^*_s$ , then $A( u_n) \rightarrow A( u)$;
\smallskip
\item[$(vi)$] if $u_n \rightharpoonup u$ in $H^s( \mathbb R^N )$ then $A( u_n-u)=A( u_n)-A( u)+o_n( 1)$.
\end{itemize}
\end{lemma}
\begin{proof}
Items $(i)$ and $(ii)$ follow directly by the definition of $\phi_{\varepsilon, u}$.

\medskip

To prove $(iii)$, let $v\in C^{\infty}_{c}(\mathbb R^{N})$; we have
\begin{eqnarray*}
\int_{\mathbb R^{N}}(-\Delta)^{\alpha/4}(\phi_{\varepsilon,u_{n}} - \phi_{\varepsilon, u})(-\Delta)^{\alpha/4}v
&=&\int_{\mathbb R^{N}} (u_{n}^{2} - u^{2})v \\
&\leq& |v|_{\infty} \Big(\int_{\textrm{supp}\, v}(u_{n}-u)^{2}\Big)^{1/2} \Big(\int_{\textrm{supp}\, v}(u_{n}+u)^{2}\Big)^{1/2} \\
&\to & 0.
\end{eqnarray*}
The conclusion then follows by density.

\medskip

The proof of  $(iv)$ is 
straightforward: 
we refer the reader  to \cite{fangzhang}.

\medskip

To show $(v)$, recall that $2<\frac{4N}{N+\alpha}< 2^*_s$. Since by assumption $\abs{u_n^{2}}_\frac{2N}{N+\alpha} \rightarrow \abs{u^2}_\frac{2N}{N+\alpha}$ and $u_n^2 \rightarrow u^2$ a.e. in $\mathbb R^N$, using the Brezis-Lieb Lemma, $u_n^2 \rightarrow u^2$ in $L^\frac{2N}{N+\alpha}\left( \mathbb R^N\right)$. But then
 using \eqref{eq:convergfacil} we get $\phi_{\varepsilon,u_{n}} \to \phi_{\varepsilon, u}$ in 
$L^{2^*_{\alpha/2}}( \mathbb R^N)$. Consequently
\[
\begin{split}
\abs{A\left( u_n\right) -A\left( u\right)} & \leq \int_{\mathbb R^{N}} \abs{\phi_{\varepsilon, u_n}u_n^2-\phi_{\varepsilon, u}u^2} \leq \int_{\mathbb R^{N}} \abs{\left( \phi_{\varepsilon, u_n}-\phi_{\varepsilon, u} \right)u_n^2} +\int_{\mathbb R^{N}} \abs{\phi_{\varepsilon, u} \left( u_n^2-u^2\right)}\\
& \leq \abs{\phi_{\varepsilon, u_n}-\phi_{\varepsilon, u}}_{2^*_{\alpha/2}}\abs{u_n^2}_\frac{2N}{N+\alpha}+\abs{\phi_{\varepsilon, u}}_{2^*_{\alpha/2}}\abs{u_n^2-u^2}_\frac{2N}{N+\alpha}
\end{split}
\]
from which we conclude.

\medskip

To prove $(vi)$, for the sake of simplicity we drop
the factor $\varepsilon^{\alpha-\theta}$ in the expression of $\phi_{\varepsilon,u,v}$. Defining
\[
\sigma := \int_{\mathbb R^{N}} \int_{\mathbb R^{N}} \frac{u^2( y)u^2( x)}{\abs{x-y}^{N-\alpha}}dydx\,,
\]
\begin{eqnarray*}
\sigma^1_n := \int_{\mathbb R^{N}} \int_{\mathbb R^{N}} \frac{u_n^2 (y)u^2( x)}{\abs{x-y}^{N-\alpha}}dydx\,, \qquad
 \sigma^2_n := \int_{\mathbb R^{N}} \int_{\mathbb R^{N}} \frac{u_n( y)u( y)u_n( x)u( x)}{\abs{x-y}^{N-\alpha}}dydx\\
\sigma^3_n := \int_{\mathbb R^{N}} \int_{\mathbb R^{N}} \frac{u_n^2( y)u_n( x)u( x)}{\abs{x-y}^{N-\alpha}}dydx  \,, \qquad
\sigma^4_n := \int_{\mathbb R^{N}} \int_{\mathbb R^{N}} \frac{u_n( y)u( y)u^2\left( x\right)}{\abs{x-y}^{N-\alpha}}dydx,
\end{eqnarray*}
it is easy to check that
\[
A( u_n-u) -A( u_n) +A( u) = 2\sigma+ 2\sigma^1_n +4\sigma^2_n-4\sigma^3_n-4\sigma^4_n.
\]
Now we claim that, whenever $u_{n}\rightharpoonup u$ in $H^{s}(\mathbb R^{N})$,
\[
\lim _{n\rightarrow \infty} \sigma^i_n = \sigma, \quad i=1,2,3,4
\]
which readily  gives the conclusion.

We prove here only the cases $i=1,2$ since the proof of the other cases is very similar.
Recall that
$$\phi_{\varepsilon, u} (x)=  \int_{\mathbb R^N}\frac{u^2\left( y\right)}{\abs{x-y}^{N-\alpha}}dy\,, \qquad 
\phi_{\varepsilon, u_{n}} (x)=  \int_{\mathbb R^N}\frac{u_{n}^2( y)}{\abs{x-y}^{N-\alpha}}dy.$$
Since $u^2\in L^{\frac{2N}{N+\alpha}}( \mathbb R^N)=L^{(2^{*}_{\alpha/2})'}(\mathbb R^{N})$   and by item $(iii)$ it holds 
$\phi_{\varepsilon,u_{n}} \to \phi_{\varepsilon,u}$ in $L^{2^{*}_{\alpha/2}}(\mathbb R^{N})$, we conclude that
\[
\sigma_{n}^{1}=\int_{\mathbb R^N} \phi_{\varepsilon, u_{n}}u^2 \rightarrow \int_{\mathbb R^N} \phi_{\varepsilon,u} u^2=\sigma
\]
and the claim is true for $i=1$.

For $i=2$ recall that
\[
 \phi_{\varepsilon, u_{n },u}(x) = \int_{\mathbb R^N}\frac{u_n\left( y\right)u\left( y\right)}{\abs{x-y}^{N-\alpha}}dy.
\]
First we show that $\phi_{\varepsilon, u_{n },u} \rightarrow \phi_{\varepsilon, u}$ a.e. in $\mathbb R^N$.
 Given $\xi >0$ and  choosing $R>1/\xi, \frac{N}{2s}<p<\frac{N}{N-\alpha}$ and $\frac{N}{N-\alpha}<q$
 (so that $2p', 2q'\in (2, 2^{*}_{s})$), we have, for large $n$:
\begin{eqnarray*}
\abs{\phi_{\varepsilon, u_{n },u}(x)-\phi_{\varepsilon,u}( x)} &\leq& \abs{u_n-u}_{L^{2p'}( B_R( x))}\abs{u}_{L^{2p'}( B_R( x))}\Big( \int_{\abs{y-x}<R}\frac{dy}{\abs{x-y}^{p( N-\alpha)}}\Big)^{1/p}\\
&+& \abs{u_n-u}_{L^{2q'}( B_R^c( x))}\abs{u}_{L^{2q'}( B_R^c( x))}\left( \int_{\abs{y-x}\geq R}\frac{dy}{\abs{x-y}^{q( N-\alpha)}}\right)^{1/q}\\
&\leq&C_{1}\xi+C_{2}\xi^{N-\alpha},
\end{eqnarray*}
concluding the pointwise convergence.
Moreover by the Sobolev embedding and using \eqref{eq:10},
\[
\abs{\phi_{\varepsilon,u_{n}, u} u_n}_2\leq \abs{\phi_{\varepsilon,u_{n}, u}}_{2^*_{\alpha/2}} \abs{u_n}_{2N/\alpha}\leq C_{1}\norm{u_n}^2 \norm{u}\leq C_{2}
\]
and therefore, up to subsequence, $\phi_{\varepsilon,u_{n}, u} u_n \rightharpoonup \phi_{\varepsilon, u}u$ in 
$L^2( \mathbb R^N)$, by \cite[Lemma 4.8]{Kav}. Since $u\in L^2( \mathbb R^N)$
\[
\sigma_{n}^{2} = \int_{\mathbb R^N} \phi_{\varepsilon,u_{n}, u} u_n u \rightarrow \int_{\mathbb R^N} \phi_{\varepsilon, u} u^2
=\sigma
\]
and the claim  is proved for $i=2$.
\end{proof}

\medskip

We introduce now the variational setting for our problem.
Let us define the Hilbert space 
\[
W_\varepsilon := \left\{ u\in H^s( \mathbb R^N):\int_{\mathbb R^N}V( \varepsilon x)u^2<\infty \right\}
\]
endowed with scalar product and (squared) norm given by
\[
(u,v) _\varepsilon:= \int_{\mathbb R^N} (-\Delta )^{s/2}u(-\Delta )^{s/2}v+  \int_{\mathbb R^N}V( \varepsilon x)uv
\]
and
\[
\norm{u}_\varepsilon^2 := \int_{\mathbb R^N} \abs{(-\Delta )^{s/2}u}^2+\int_{\mathbb R^N}V( \varepsilon x)u^2.
\]
Then it  is standard to see that the critical points of the $C^{2}$ functional (see Lemma \ref{propiedades} $(iv)$)
\[
I_\varepsilon ( u):= \frac{1}{2}\int_{\mathbb R^N} \abs{(-\Delta )^{s/2}u}^2+\frac{1}{2}\int_{\mathbb R^N}V( \varepsilon x)u^2+\frac{1}{4}\int_{\mathbb R^N} \phi_{\varepsilon, u}u^2-\int_{\mathbb R^N}F( u),
\]
on $W_{\varepsilon}$ are weak solutions of problem \eqref{prob2}.

By defining
$$
{\cal N}_{\varepsilon}:=\Big\{u\in W_{\varepsilon}\setminus \{0\}: J_{\varepsilon}(u)=0 \Big\},
$$
 where $$J_{\varepsilon}( u):=I'_\varepsilon ( u) [ u]=\norm{u}_\varepsilon^2+\int_{\mathbb R^N} 
 \phi_{\varepsilon, u}u^2-\int_{\mathbb R^N}f( u)u,$$
we have, by standard arguments:
\begin{lemma}
For every $u \in {\cal N}_{\varepsilon}$, $J'_\varepsilon ( u) [ u]<0$ and there are positive constants $h_\varepsilon, k_\varepsilon$ such that $\norm{u}_\varepsilon \geq h_\varepsilon, I_\varepsilon ( u) \geq k_\varepsilon$. Furthermore, ${\cal N}_{\varepsilon}$ is diffeomorphic to the set
\[
\mathcal S_{\varepsilon}:=\left\{ u\in W_\varepsilon: \norm{u}_\varepsilon=1,\ u>0 \text{ a.e.} \right\}.
\]
\end{lemma}
$\mathcal N_{\varepsilon}$ is the {\sl Nehari manifold} associated to $I_{\varepsilon}.$
By the assumptions on $f$,  the functional $I_{\varepsilon}$ has the Mountain Pass geometry. This is standard but
we give the easy proof for completeness. \smallskip
\begin{itemize}
\item[(MP1)] $I_{\varepsilon}(0)=0;$ \smallskip
\item[(MP2)] since, for every $\xi>0$ there exists $M_{\xi}>0$ such that $F(u)\leq \xi u^{2} + M_{\xi}|u|^{q_{0}+1}$, we have
\begin{eqnarray*}
I_{\varepsilon}(u) &\geq& \frac{1}{2}\|u\|^{2}_{\varepsilon} -\int_{\mathbb R^{N}} F(u) \\
&\geq& \frac{1}{2}\|u\|^{2}_{\varepsilon} -\xi C_{1} \|u\|_{\varepsilon}^{2}-M_{\xi}C_{2} \|u\|_{\varepsilon}^{q_{0}+1}
\end{eqnarray*}
and we conclude $I_\varepsilon$ has a strict local minimum at $u=0$; \smallskip

\item[(MP3)] finally, since  \eqref{f_{4}} implies $F( t)\geq Ct^K$ for $t>0$, with $K>4$ (and less then $q_{0}+1$),
fixed  ${v}\in C^{\infty}_{c}(\mathbb R^{N}), v>0$ 
we have
\begin{align*}
I_\varepsilon ( t v)&= \frac{t^2}{2}\norm{v}^2_\varepsilon+\frac{t^{4}}{4}\int_{\mathbb R^N} \phi_{\varepsilon, v}v^2 -\int_{\mathbb R^N}F( tv)\\
& \leq \frac{t^2}{2}\norm{v}^2_\varepsilon+\frac{t^4}{4}\int_{\mathbb R^N} \phi_{\varepsilon, v}v^2 -Ct^K\int_{\mathbb R^N}v^K 
\end{align*}
concluding that the functional is negative for suitable large $t$.
\end{itemize}

Then denoting with
$$
c_{\varepsilon}:=\inf_{\gamma\in \mathcal H_{\varepsilon}}\ \sup_{t\in[0,1]} I_{\varepsilon}(\gamma(t)),
\qquad \mathcal H_{\varepsilon}=\Big\{\gamma\in C([0,1], W_{\varepsilon}): \gamma(0)=0, I_{\varepsilon}(\gamma(1))<0\Big\}
$$
the Mountain Pass level, and with 
$$m_{\varepsilon}:= \inf_{u\in {\cal N}_{\varepsilon}}I_{\varepsilon}(u)$$
the {\sl ground state level},
it holds,   in a standard way, that
\begin{equation}\label{eq:defce}
 c_{\varepsilon}= m_{\varepsilon}=
 \inf_{u\in W_{\varepsilon}\setminus \{0\}}\sup_{t\geq0}I_{\varepsilon}(tu).
\end{equation}

\medskip

\medskip

It is known that for ``perturbed'' problems a major role is played by the problem at infinity that we now introduce.

\subsection{The problem at ``infinity''}
Let us consider the ``limit'' problem  (the autonomous problem)
associated to \eqref{prob2}, that is
\begin{equation}\label{prob4}\tag{$A_{\mu}$}
\left\{
	\begin{array}{l}
		(-\Delta )^{s}u+\mu u=f(u) \smallskip\\
        u\in H^s( \mathbb R^N)\\
	\end{array}
\right.
\end{equation}
where $\mu>0$ is a constant. The solutions are critical points of the functional
$$
E_\mu (u)= \frac{1}{2}\int_{\mathbb R^N} \abs{(-\Delta )^{s/2}u}^2+\frac{\mu}{2}\int_{\mathbb R^N}u^2-\int_{\mathbb R^N}F( u).
$$
in $ H^s( \mathbb R^N)$.
Denoting with $H^s_{\mu}( \mathbb R^N)$ simply the space $H^{s}(\mathbb R^{N})$ endowed with the (equivalent squared) norm 
\[
\norm{u}_{H^{s}_{\mu}}^2:=|(-\Delta )^{s/2}u|_{2}^2+\mu |u|^2_{2},
\]
by the assumptions of the nonlinearity $f$, it is easy to see that 
the functional $E_{\mu}$ has the Mountain Pass geometry with Mountain Pass level
\begin{equation*}\label{eq:mpinfinito}
 c^{\infty}_{\mu}:=\inf_{\gamma\in \mathcal H_{\mu}}\ \sup_{t\in[0,1]} E_{\mu}(\gamma(t)),
\qquad \mathcal H_{\mu}:=\Big\{\gamma\in C([0,1], H^{s}_{\mu}(\mathbb R^{N})): \gamma(0)=0, E_{\mu}(\gamma(1))<0\Big\}.
\end{equation*}
Introducing the set 
\[
{\cal M}_{\mu}:=\Big\{u\in H^{s}(\mathbb R^N)\setminus\{0\}: \|u\|^{2}_{H^{s}_{\mu}}
=\int_{\mathbb R^N}f(u)u \Big\}
\]
it is standard to see that 
\begin{itemize}
\item ${\cal M}_{\mu}$ has a structure of differentiable manifold (said  the {\sl Nehari manifold} associated to $E_{\mu}$),
\item ${\cal M}_{\mu}$ is bounded away from zero and radially homeomorfic to the unit sphere,
\item  the mountain pass value $c^{\infty}_{\mu}$ coincide with the 
{\sl ground state level}
\[
m^{\infty}_{\mu}:=
\inf_{u\in \mathcal M_{\mu}}E_{\mu}(u)>0.
\]
\end{itemize}
The symbol ``$\infty$'' in the notations  is just to recall we are dealing with the limit problem. In the sequel
we will mainly deal with $\mu=V_{0}$ and $\mu=V_{\infty}$ (whenever this last one is finite).
Of course the inequality 
$$m_{\varepsilon}\geq m^{\infty}_{V_{0}}$$
holds.

\medskip

\section{Compactness properties for $I_{\varepsilon}, E_{\mu}:$ \\ existence of a ground state solution}\label{compattezza}
We begin by showing the boundedness of the Palais-Smale sequences for $ E_\mu$ in $H^{s}_{\mu}(\mathbb R^{N})$
and $I_\varepsilon$ in $W_{\varepsilon}.$
 Let $\{ u_n\} \subset H^{s}_{\mu}(\mathbb R^{N})$ be a Palais-Smale sequence for $E_\mu$, that is, $\abs{E_\mu (u_n) } \leq C$ and $E_\mu' ( u_n) \rightarrow 0$. Then, for large $n$,
\begin{align*}
C+\norm{u_n}_{H^{s}_{\mu}} > E_\mu ( u_n)-\frac{1}{K}E_\mu' ( u_n) [ u_n] &= \left( \frac{1}{2}-\frac{1}{K}\right) \norm{u_n}^2_{H^{s}_\mu}+\frac{1}{K}\int_{\mathbb R^N} ( f( u_n)u_n-KF( u_n) )\\
&\geq \left( \frac{1}{2}-\frac{1}{K}\right) \norm{u_n}^2_{H^{s}_{\mu}},
\end{align*}
and thus $\{u_n\}$ is bounded. Similarly we conclude for $I_\varepsilon$, using that
\begin{equation*}
\begin{split}
I_\varepsilon ( u_n)-\frac{1}{K}I'_\varepsilon ( u_n) [ u_n] & =\left( \frac{1}{2}-\frac{1}{K}\right) \norm{u_n}^2_\varepsilon+
\left( \frac{1}{4}-\frac{1}{K}\right) \int_{\mathbb R^N}\phi_{\varepsilon, u_n}u_n^2  +\frac{1}{K}\int_{\mathbb R^N} ( f( u_n)u_n-
KF( u_n) ) \\
&\geq \left( \frac{1}{2}-\frac{1}{K}\right) \norm{u_n}^2_\varepsilon.
\end{split}
\end{equation*}

In order to prove compactness, some  preliminary work is needed.
Let us recall the following Lions type lemma, whose proof can be found in
\cite[Lemma 2.3]{GPM}.
\begin{lemma}\label{lionslemma}
If $\{u_{n}\}$ is bounded in $H^{s}(\mathbb R^{N})$ and for some $R>0$ and $2\leq r< 2_{s}^{*}$ we have $$\sup_{x\in \mathbb R^{N}}\int_{B_R(x)}|u_{n}|^{r}\to 0\quad  \text{ as }\quad  n\to \infty,$$ then $u_{n}\to 0$ in $L^{p}(\mathbb R^{N})$ for $2< p <2_{s}^{*}$.
\end{lemma}
Then we can prove the following
\begin{lemma}\label{gio5}
Let $\{u_{n}\}\subset W_{\varepsilon}$ be bounded and such that $I_{\varepsilon}'(u_{n})\to 0$.
Then we have either
\begin{itemize}
\item[a)]$u_{n} \rightarrow 0$ in $W_{\varepsilon}$, or \smallskip
\item[b)] there exist a sequence $\{y_{n}\} \subset {\mathbb R^N}$ and constants $R, c > 0$ such that
$$
\liminf_{n\rightarrow +
\infty}\int_{B_{R}(y_{n})}u_{n}^{2} \geq c > 0.
$$
\end{itemize}
\end{lemma}
\begin{proof}
Suppose that b) does not occur. Using Lemma \ref{lionslemma}  it follows
$$
u_{n} \rightarrow 0 \,\,\, \mbox{in}\,\,\, L^{p}({\mathbb R^N})\,\,\, \mbox{for}\,\,\, p \in (2, 2^{*}_{s}).
$$
Using \eqref{terminof},
the boundedness of $\{u_{n}\}$  in $L^{2}(\mathbb R^{N})$ and  the fact that 
$u_{n} \rightarrow 0$ in $L^{q_{0}+1}(\mathbb R^{N})$, 
we conclude that
$$
\int_{\mathbb R^N}f(u_{n})u_{n}\rightarrow 0.
$$
Finally, since $$\norm{u_n}^2_\varepsilon-\int_{\mathbb R^N}f(u_{n})u_{n}\leq \norm{u_n}^2_\varepsilon +\int_{\mathbb R^N}\phi _{\varepsilon, u_n}u_n^2 -\int_{\mathbb R^N}f(u_{n})u_{n}= I'_{\varepsilon}(u_{n})[u_{n}]=o_{n}(1),$$ it follows that $u_{n} \rightarrow 0$ in $ W_{\varepsilon}$.
\end{proof}

\medskip

In the rest of the paper we assume, without loss of generality, that $0\in M$, that is, $V(0)=V_{0}$.

\begin{lemma}\label{gio9}
Assume that $V_{\infty}<\infty$ and let $\{v_{n}\}\subset W_{\varepsilon}$ be a $(PS)_{d}$ sequence for  $I_{\varepsilon}$ such that $v_{n}\rightharpoonup 0$ in $W_{\varepsilon}$. Then $$v_{n}\not\rightarrow 0\ \  \text{in}\ \ W_{\varepsilon}\ \Longrightarrow \ d \geq
m^{\infty}_{V_{\infty}}.$$
\end{lemma}
\begin{proof}
Observe, preliminarly, that by condition \eqref{V} it follows that
\begin{equation}\label{eq:conseqV1}
 \forall \xi>0 \ \exists \widetilde R=\widetilde  R_{\xi}>0 : \ \  V(\varepsilon x)> V_{\infty} - \xi, \qquad \forall x\notin B_{\widetilde  R}.
\end{equation}

Let $\{t_{n}\}\subset (0,+\infty )$ be such that $\{t_{n}v_{n}\}\subset {\cal M}_{V_{\infty}}$. We start by showing the following

\medskip

\noindent \textbf{Claim:} {\it The sequence $\{t_{n}\}$ satisfies $\limsup_{n\rightarrow \infty}t_{n} \leq 1$. }

\medskip

\noindent Supposing by contradiction that the claim does not hold, there exists $\delta> 0$ and a subsequence still denoted by $\{t_{n}\}$, such that
\begin{eqnarray}\label{limsup}
t_{n} \geq 1 + \delta \quad \mbox{for all} \quad n \in \mathbb  N.
\end{eqnarray}
Since $\left\{v_n \right\}$ is a bounded $(PS)_{d}$ sequence for $I_\varepsilon$, $I'_{\varepsilon}(v_{n})[v_{n}]=o_{n}(1)$, that is,
\[
\norm{v_n}^2_\varepsilon +\int_{\mathbb R^n}\phi_{\varepsilon, v_n}v_n^2= \int_{\mathbb R^n} f( v_n)v_n+o_n\left( 1\right).
\]
Moreover, since $\{t_{n}v_{n}\}\subset {\cal M}_{V_{\infty}}$, we get
$$
\norm{t_n v_n}^2_{H^{s}_{V_\infty}}=\int_{\mathbb R^N}f(t_{n}v_{n})t_{n}v_{n}.
$$
These equalities imply that
\[
\int_{\mathbb R^N}\left(\frac{f(t_{n}v_{n})}{t_{n}}- f(v_{n})\right)v_{n}=\int_{\mathbb R^N}(V_{\infty} - V(\varepsilon x))v_{n} ^{2}-\int_{\mathbb R^N}\phi_{\varepsilon, v_n}v_n^2+o_{n}(1),
\]
and thus
\begin{equation}\label{bo}
\int_{\mathbb R^N}\Big(\frac{f(t_{n}v_{n})}{t_{n}}- f(v_{n})\Big)v_{n} \leq \int_{\mathbb R^N}(V_{\infty} - V(\varepsilon x))v_{n} ^{2}+o_{n}(1).
\end{equation}
 Using \eqref{eq:conseqV1}, the fact that
$v_{n}\rightarrow 0$ in $L^{2}(B_{\widetilde  R})$ and that $\{v_{n}\}$ is bounded in $W_{\varepsilon}$, 
let us say by some constant $C>0$,
we deduce by \eqref{bo}
\begin{eqnarray}\label{mudou}
\forall \xi>0: \int_{\mathbb R^N}\left(\frac{f(t_{n}v_{n})}{t_{n}} -f(v_{n})\right)v_{n} \leq \xi  C +o_{n}(1).
\end{eqnarray}
Since $v_{n}\not\rightarrow 0$ in $W_{\varepsilon}$, we may invoke Lemma \ref{gio5} to obtain $\{y_{n}\}\subset \mathbb R^{N}$ and ${R},{c} >0$ such that
\begin{eqnarray}\label{1eq4}
\int_{B_{ R}( y_{n})} v_{n}^{2}\geq{c}.
\end{eqnarray}
Defining $\check{v}_{n} := v_{n}(\cdot+ y_{n})$, we may suppose that, up to a subsequence,
$$
\check{v}_{n}\rightharpoonup \check{v} \,\,\, \mbox{in}\,\,\, H^{s}(\mathbb R^{N})
$$
and, in view of (\ref{1eq4}), there exists a subset $\Omega \subset \mathbb R^{N}$ with positive measure such that $\check{v}>0$ in $\Omega$. 
By \eqref{f_{5}} and  \eqref{limsup}, \eqref{mudou} becomes
\begin{eqnarray*}
0 < \int_{\Omega}\left(\frac{f((1 + \delta)\check{v}_{n})}{(1+ \delta)\check{v}_{n}}-\frac{f(\check{v}_{n})}{\check{v}_{n}}\right)\check{v}_{n}^{2} \leq \xi C + o_{n}(1).
\end{eqnarray*}
Now passing to the limit  and applying Fatou's Lemma, it follows that, for every $\xi>0$
$$
0<\int_{\Omega}\biggl[\frac{f((1 + \delta)\check{v})}{(1 +\delta)\check{v}}-\frac{f(\check{v})}{\check{v}}\biggl]\check{v}^{2}\leq \xi C, 
$$
which is  absurd and proves the claim.

\medskip 

Now we distinguish two cases.

\medskip

\noindent \textbf{Case 1:} $ \limsup_{n \to \infty}t_{n}=1.$ \smallskip

Up to subsequence we can assume  that $t_{n}\rightarrow 1$. 
We have,
\begin{equation}\label{1eq5}
d+ o_n( 1)=I_{\varepsilon}(v_{n}) \geq m^{\infty}_{ V_\infty } + I_\varepsilon ( v_n) -E_{V_\infty}( t_nv_n) .
\end{equation}
Moreover,
\begin{equation*}
\begin{split}
I_{\varepsilon}(v_{n})- E_{V_{\infty}}(t_{n}v_{n})&=\frac{(1-t_{n}^{2})}{2}\int_{\mathbb R^{N}}|(-\Delta)^{s/2} v_{n}|^{2}+
\frac{1}{2}\int_{\mathbb R^{N}}(V(\varepsilon x)-t_{n}^2V_{\infty})v_{n}^{2}\\
& \quad +\frac{1}{4}\int_{\mathbb R^n}\phi_{\varepsilon, v_n}v_n^2+\int_{\mathbb R^{N}}(F(t_{n}v_{n}) - 
F(v_{n})),
\end{split}
\end{equation*}
and due to the boundedness of $\{v_{n}\}$ we get, for every $\xi>0$,
\begin{eqnarray*}
I_{\varepsilon}(v_{n}) - E_{V_{\infty}}(t_{n}v_{n}) \geq \ o_{n}(1) -C \xi + \int_{\mathbb R^{N}}(F(t_{n}v_{n}) - F(v_{n})),
\end{eqnarray*}
where we have used again \eqref{eq:conseqV1}.
By the Mean Value Theorem,
$\int_{\mathbb R^{N}}(F(t_{n}v_{n}) - F(v_{n}))= o_{n}(1),$
therefore \eqref{1eq5} becomes
\begin{eqnarray*}
d + o_{n}(1) \geq m^{\infty}_{V_{\infty}} - C \xi + o_{n}(1),
\end{eqnarray*}
and taking the limit in $n$, by the arbitrariness of $\xi$,  we deduce $d \geq m^{\infty}_{{V_{\infty}}}$.

\medskip

\noindent \textbf{Case 2:} $\limsup_{n \to \infty}t_{n}=t_{0} <1$. \smallskip

We can assume $t_{n} \to t_{0} \,\,\, \mbox{and} \,\,\, t_{n} < 1$.
Since $t\mapsto \frac{1}{4}f(t)t-F(t)$ is increasing in $\left( 0, \infty\right)$,
\begin{align}\label{eq:case2_1}
m^{\infty}_{V_{\infty}} \leq E_{V_\infty}( t_n v_n) &= \int_{\mathbb R^{N}}\left(\frac{1}{2}f(t_{n}v_{n})t_{n}v_{n}-F(t_{n}v_{n})\right) \nonumber \\
&= \int_{\mathbb R^{N}}\frac{1}{4}f(t_{n}v_{n})t_{n}v_{n}+\int_{\mathbb R^{N}}\left(\frac{1}{4}f(t_{n}v_{n})t_{n}v_{n}-F(t_{n}v_{n})\right)\nonumber \\
&= \frac{1}{4}\norm{t_n v_n}^2_{H^{s}_{V_\infty}}+\int_{\mathbb R^{N}}\left(\frac{1}{4}f(t_{n}v_{n})t_{n}v_{n}-F(t_{n}v_{n})\right)\nonumber \\
&\leq \frac{1}{4}\norm{t_n v_n}^2_{H^{s}_{V_\infty}}+\int_{\mathbb R^{N}}\left(\frac{1}{4}f(v_{n})v_{n}-F(v_{n})\right).
\end{align}
But
\begin{equation}\label{case2_2}
\norm{t_n v_n}^2_{V_\infty}\leq \int_{\mathbb R^N}\abs{(-\Delta)^{s/2} v_n}^2+\int_{\mathbb R^N} t_n^2V_\infty  v_n^2.
\end{equation}
Again by \eqref{eq:conseqV1}, given $\xi>0,$
\[
t_n^2V_\infty-\xi <V_\infty-\xi < V( \varepsilon x)\quad \text{ for } {x}\notin B_{\widetilde R}
\]
and hence
\begin{align*}
\int_{\mathbb R^N} t_n^2V_\infty  v_n^2&\leq \int_{B_{\widetilde R}}V_\infty v_n^2+\int_{\abs{x}\geq \widetilde R}V( \varepsilon x ) v_n^2+\int_{\abs{x}\geq \widetilde R} \xi v_n^2\\
&\leq o_n( 1)+\int_{\mathbb R^N}V( \varepsilon x ) v_n^2+C\xi.
\end{align*}
From this and \eqref{case2_2} we have
\[
\norm{t_n v_n}^2_{H^{s}_{{V_\infty}}}\leq \norm{v_n}_\varepsilon ^2+C\xi +o_n( 1).
\]
Therefore, using \eqref{eq:case2_1}
\begin{align*}
m^{\infty}_{ V_\infty } & \leq  \frac{1}{4}\norm{v_n}^2_\varepsilon+\int_{\mathbb R^N} \left( \frac{1}{4}f( v_n)v_n -F( v_n)\right)  +C\xi +o_n( 1)\\
&= I_\varepsilon (v_n)-\frac{1}{4}I'_\varepsilon ( v_n) [ v_n]  +C\xi +o_n( 1)\\
& = d+C\xi +o_{n}(1).
\end{align*}
concluding the proof.
\end{proof}
\begin{proposition}\label{gio11}
The functional $I_{\varepsilon}$  in $W_{\varepsilon}$ satisfies the $(PS)_{c}$ condition \smallskip
\begin{itemize}
\item[1.] at any level $c < m^{\infty}_{V_{\infty}}$, if $V_{\infty}<\infty$, \medskip
\item[2.] at any level $c\in\mathbb R$, if $V_{\infty}=\infty$.
\end{itemize}
\end{proposition}
\begin{proof}
Let $\{u_{n}\}\subset W_{\varepsilon}$ be such that $I_{\varepsilon}(u_{n})\rightarrow c$ and $I'_{\varepsilon}(u_{n})\rightarrow 0$. We have already seen that $\{u_{n}\}$ is bounded in $W_{\varepsilon}$. Thus there exists $u \in W_{\varepsilon}$ such that, up to a subsequence, $u_{n}\rightharpoonup u$ in $W_{\varepsilon}$.
Note that $I'_{\varepsilon}(u)=0$, since
 by Lemma \ref{propiedades} $(iv)$, we have for every $w\in W_{\varepsilon}$
\[
\left(u_n,w\right) _\varepsilon \rightarrow (u,w) _\varepsilon, \quad A'( u_n) [ w] \rightarrow A'( u) [ w] \quad \text{and } \int_{\mathbb R^{N}} f( u_n) w \rightarrow \int_{\mathbb R^{N}} f( u) w.
\]
Defining $v_{n}:= u_{n} - u$, we  have that $\int_{\mathbb R^{N}} F(v_{n})=\int_{\mathbb R^{N}}F(u_{n})-\int_{\mathbb R^{N}}F(u)+o_n(1)$ (see  \cite{Alves3}) and 
by Lemma \ref{propiedades} $(vi)$, we have $ A(v_n) =A( u_n)-A( u)+o_n( 1)$; 
hence arguing as in \cite{Alvesgio1}, we obtain also
\begin{equation}\label{eq:0}
I'_{\varepsilon}(v_{n})\rightarrow 0.
\end{equation}
 Moreover
\begin{equation}\label{cd}
I_{\varepsilon}(v_{n})=I_{\varepsilon}(u_{n})-I_{\varepsilon}(u)+o_{n}(1)=c - I_{\varepsilon}(u)+o_{n}(1)=:d +o_{n}(1)
\end{equation}
and \eqref{eq:0} and  \eqref{cd} show that $\{v_{n}\}$ is a $(PS)_{d}$ sequence. By \eqref{f_{4}},
\begin{multline*}
I_{\varepsilon}(u)= I_{\varepsilon}(u)- \frac{1}{4}I'_{\varepsilon}(u)[u]=\frac{1}{4}\norm{u}^2_\varepsilon +\int_{\mathbb R^{N}}
\Big(\frac{1}{4}f(u)u - F(u)\Big) \\
\geq \frac{1}{4}\int_{\mathbb R^{N}}\Big(f(u)u - 4F(u)\Big) \geq 0
\end{multline*}
and then coming back in \eqref{cd} we have
\begin{equation}\label{eq:siusa}
d\leq c.
\end{equation}
Then,  \smallskip

1. if $V_{\infty}<\infty$, and $c<m_{V_{\infty}}^{\infty}$, by \eqref{eq:siusa} we obtain
$$
d\leq c < m^{\infty}_{V_{\infty}}.
$$
It follows from Lemma \ref{gio9} that $v_{n}\rightarrow 0$, that is $u_{n}\rightarrow u$ in $W_{\varepsilon}$. \medskip

 2. If $V_{\infty}= \infty$, by the compact imbedding $W_{\varepsilon} \hookrightarrow\hookrightarrow L^{r}(\mathbb R^N), 2 \leq r < 2^{*}_{s}$,
up to a subsequence, $v_{n}\rightarrow 0$ in $L^{r}(\mathbb R^N)$ and since $I'_{\varepsilon}(v_{n})\rightarrow 0$, we have
\begin{equation}\label{eq:anche}
I'_{\varepsilon}(v_n) [ v_n] =\norm{v_n}_\varepsilon ^2+\int_{\mathbb R^{N}} \phi_{\varepsilon, v_n}v_n^2-\int_{\mathbb R^N}f(v_n)v_n=o_n(1).
\end{equation}

By Lemma \ref{propiedades} $(v)$,
$A(v_n)=\int_{\mathbb R^{N}} \phi_{\varepsilon, v_n}v_n^2=o_n(1)$, and since by  \eqref{terminof}
it holds again $\int_{\mathbb R^{N}}f(v_{n})v_{n}=o_{n}(1)$, we have by \eqref{eq:anche}
$\norm{v_n}_\varepsilon ^2=o_{n}(1)$, that is $u_{n} \to u$ in $W_{\varepsilon}$.

\medskip

The proof is thereby complete.
\end{proof}
As a consequence it is standard to prove that 
\begin{proposition}\label{gio13}
The functional $I_{\varepsilon}$ restricted to $\cal{N}_{\varepsilon}$ satisfies the $(PS)_{c}$ condition 
\begin{itemize}
\item[1.] at any level $c <m_{V_{\infty}}^{\infty}$, if $V_{\infty}<\infty$, \medskip
\item[2.]  at any level $c\in\mathbb R$, if $V_{\infty}=\infty$.
\end{itemize}
Moreover,  the constrained critical points of the functional $I_{\varepsilon}$ on ${\cal N}_{\varepsilon}$ are critical points of $I_{\varepsilon}$ in $W_{\varepsilon}$, hence solution of \eqref{prob2}.
\end{proposition}
Let us recall the following result (see \cite[Lemma 6]{gaetanogiovany}) concerning 
problem \eqref{prob4}.
\begin{lemma}[Ground state for the autonomous problem]\label{gio16}
Let $\{u_{n}\} \subset {\cal{M}}_{\mu}$ be a sequence satisfying $E_{\mu}(u_{n})\rightarrow m^{\infty}_{\mu}$. Then, up to subsequences the following alternative holds:
\begin{itemize}
\item[a)] $\{u_{n}\}$ strongly converges in $H^{s}(\mathbb R^{N})$; \medskip
\item[b)] there exists a sequence $\{\tilde{y}_{n}\}\subset \mathbb R^{N}$ such that $u_{n}(\cdot+\tilde{y}_{n})$ strongly converges in $H^{s}(\mathbb R^{N})$.
\end{itemize}
In particular, there exists a minimizer $\mathfrak w_{\mu}\geq0$ for $m^{\infty}_{\mu}$.
\end{lemma}
Now we can prove the existence of a ground state for our problem. Assumption \eqref{H} is tacitly assumed.
\begin{theorem}\label{gio12}
Suppose that $f$ verifies \eqref{f_{1}}-\eqref{f_{5}} and $V$ verifies \eqref{V}. Then there exists a ground state solution $\mathfrak u_{\varepsilon}\in W_{\varepsilon}$ of \eqref{prob2}, \smallskip
\begin{itemize}
\item[1.] for every  $\varepsilon\in (0,\bar\varepsilon]$, for some $\bar\varepsilon>0$, if $V_{\infty}<\infty$; \medskip
\item[2.] for every $\varepsilon>0$, if  $V_{\infty}=\infty$.
\end{itemize} 
\end{theorem}
\begin{proof}
Since the functional $I_{\varepsilon}$ has the geometry of the Mountain Pass Theorem in $W_{\varepsilon}$
there exists $\{u_{n}\}\subset  W_{\varepsilon}$ satisfying
$$
I_{\varepsilon}(u_{n})\rightarrow c_{\varepsilon}\,\,\, \mbox{and}\,\,\,
I'_{\varepsilon}(u_{n})\rightarrow 0.
$$

\medskip

 1. If $V_{\infty} < \infty$, 
in virtue of  Proposition \ref{gio11}, we have only to show that 
$c_{\varepsilon} < m^{\infty}_{V_{\infty}}$ for every positive $\varepsilon$ smaller than a certain $\bar \varepsilon$. 

Let $\mu \in (V_{0}, V_{\infty})$, so that 
\begin{equation}\label{contradiz}
m^{\infty}_{V_{0}}<m^{\infty}_{\mu}<m^{\infty}_{V_{\infty}}.
\end{equation}
For $r>0$ let $\eta_{r}$ be a smooth cut-off function in $\mathbb R^{N}$ which equals $1$ on $B_{r}$ and with support in $B_{2r}$. Let $w_{r}:=\eta_{r} \mathfrak w_{\mu}$ and $s_{r}>0$ such that $s_{r}w_{r}\in \mathcal M_{\mu}$. If it were, for every $r>0: E_{\mu}(s_{r}w_{r})\geq m^{\infty}_{V_{\infty}}$, since $w_{r}\to \mathfrak w_{\mu}$ in $H^{s}(\mathbb R^{N})$ for $r\to +\infty$, we would have $s_{r}\to 1$ and then
$$m^{\infty}_{V_{\infty}}\leq\liminf_{r\to+\infty}E_{\mu}(s_{r}w_{r})=E_{\mu}(\mathfrak w_{\mu})=m^{\infty}_{\mu}$$
which contradicts \eqref{contradiz}. This means that there exists $\overline r>0$ such that $\omega:= s_{\bar r} w_{\bar r}\in\mathcal M_{\mu}$ satisfies 
\begin{equation}\label{eq:fundamental}
E_{\mu}(\omega)<m^{\infty}_{V_{\infty}}.
\end{equation}
Given $\varepsilon >0$, let $t_\varepsilon>0$ the number such that $t_\varepsilon \, \omega \in \mathcal N_\varepsilon$. Therefore
\[
t_\varepsilon ^2 \norm{\omega}^2_\varepsilon +t_\varepsilon ^4\int_{\mathbb R^N} \phi_{\varepsilon, \omega}\, \omega ^2=t_\varepsilon \int_{\mathbb R^N} f( t_\varepsilon \, \omega ) \omega
\]
implying that
\begin{equation}\label{desigualdad1}
\frac{\norm{\omega}^2_\varepsilon}{t_\varepsilon ^2}+\int_{\mathbb R^N} \phi_{\varepsilon, \omega}\, \omega ^2\geq \int_{B_{\overline r}} \frac{f\left( t_\varepsilon \, \omega \right)}{( t_\varepsilon \, \omega ) ^3} \omega ^4.
\end{equation}
Now we claim that there exists $T>0$ such that $\limsup_{\varepsilon \to 0^+} t_\varepsilon \leq T$. 
If by contradiction there exists 
$\varepsilon _n \to 0^+$ with $t_{\varepsilon _n} \to \infty$, then by \eqref{desigualdad1} and \eqref{f_{5}} we have
\begin{equation}\label{desigualdad2}
\frac{\norm{\omega}^2_{\varepsilon_n}}{t_{\varepsilon_n} ^2}+\int_{\mathbb R^N} \phi_{\varepsilon_n, \omega}\, \omega ^2\geq \frac{f( t_{\varepsilon_n} \, \omega ( \overline x ) )}{( t_{\varepsilon_n} \, \omega ( \overline x ) ) ^3} \int_{B_{\overline r}}  \omega ^4,
\end{equation}
where $\omega ( \overline x ):= \min_{ B_{\overline r}} \omega \left( x \right)$. 
The absurd is achieved by passing to the limit in $n$, since by \eqref{f_{4}}  the right hand side of \eqref{desigualdad2} tends to $\infty$, while the left hand side tends to $0$.

Then there exists $\varepsilon_{1}>0$ such that
\begin{equation}\label{desigualdad3}
\forall \varepsilon\in (0, \varepsilon_{1}]: \quad t_\varepsilon\in (0,T].
\end{equation}
Condition \eqref{V} implies also that there exists  some $\varepsilon _2>0$ such that
\begin{equation}\label{desigualdad4}
\forall \varepsilon\in (0,\varepsilon_{2}]: \quad V(\varepsilon x)\leq \frac{V_0 +\mu}{2},\,\,\, \mbox{ for all }\,\,\,x \in \mbox{supp}\, \omega.
\end{equation}
Finally let $$\varepsilon_{3}:= \left(\frac{(\mu - V_{0})|\omega|_{2}^{2}}{C_{\textrm e} \, T^{2}\|\omega\|^{4}} \right)^{1/(\alpha-\theta)},$$
where $ C_{\textrm e}$ is the same constant appearing in \eqref{eq:esaseusa}, hence in particular
\begin{equation}\label{desigualdad4A}
\forall \varepsilon\in (0,\varepsilon_{3}]: \quad
\int_{\mathbb R^{N}} \phi_{\varepsilon,\omega} \omega^{2}  \leq \varepsilon^{\alpha-\theta}  C_{\textrm e} \|\omega\|^{4}\quad  \text{and} \quad T^2 \varepsilon ^{\alpha -\theta} C_{\textrm e}\norm{\omega}^{4} \leq (\mu -V_0)\int_{\mathbb R^N} \omega ^2. 
\end{equation}
Let $\bar \varepsilon:= \min\{\varepsilon_{1}, \varepsilon_{2}, \varepsilon_{3}\}$.  
By using \eqref{desigualdad3}-\eqref{desigualdad4A} we have, for every $\varepsilon\in 
(0,\bar \varepsilon]$:
$$
\int_{\mathbb R^N} V(\varepsilon x) \omega ^2+ \frac{t_\varepsilon ^2}{2} \int_{\mathbb R^N} \phi_{\varepsilon, \omega}\, \omega^2 
 \leq  \frac{V_{0}+\mu}{2} |\omega|_{2}^{2} + \frac{1}{2}T^{2} \varepsilon^{\alpha-\theta} C_{\textrm e}\|\omega\|^{4}\leq  
\mu \int_{\mathbb R^N} \omega ^2, 
$$
from which we infer $ I_{\varepsilon}(t_{\varepsilon} \omega )\leq E_{\mu}(t_{\varepsilon}\omega)$.
%
Then by  \eqref{eq:defce} and 
\eqref{eq:fundamental},
\[
c_{\varepsilon}\leq  I_\varepsilon ( t_\varepsilon\, \omega ) \leq E_\mu \left( t_\varepsilon\, \omega \right) \leq E_\mu ( \omega)<m_{V_{\infty}}^{\infty}. 
\]
which concludes the proof in this case. \medskip

 2. If $V_{\infty} = \infty$, by Proposition \ref{gio11}, $\{u_{n}\}$ strongly converges to some $\mathfrak u_{\varepsilon}$ in $H^{s}(\mathbb R^{N})$, which satisfies
$$
I_{\varepsilon}(\mathfrak u_{\varepsilon})=c_{\varepsilon}\,\,\, \mbox{and}\,\,\,
I'_{\varepsilon}(\mathfrak u_{\varepsilon})=0.
$$
and $\mathfrak u_{\varepsilon}$ is the ground state we were looking for.
\end{proof}
\section{Proof of Theorem \ref{th:main}}\label{sec:bary}
In this Section we introduce the barycenter map in order to study the ``topological complexity''
of suitable sublevels of the functional $I_{\varepsilon}$ in the Nehari manifold. Let us start with the following
\begin{proposition}\label{gio17}
Let $\varepsilon_{n}\rightarrow 0^+$ and $u_n \in {\cal{N}}_{\varepsilon_{n}}$ be such that 
$I_{\varepsilon_{n}}(u_{n})\rightarrow m^{\infty}_{V_{0}}$. Then there exists a sequence $\{\tilde{y}_{n}\}\subset \mathbb R^{N}$ such that $u_{n}(\cdot+\tilde{y}_{n})$ has a convergent subsequence in $H^{s}(\mathbb R^{N})$. Moreover, up to a subsequence, $y_{n}:=\varepsilon_{n}\tilde{y}_{n}\rightarrow y \in M$.
\end{proposition}
Recall that $M$ is the set where $V$ achieves the minimum $V_{0}$.
\begin{proof}
We begin by showing that  $ \left\{ u_n\right\}$ is  bounded  in $H^s_{V_{0}}( \mathbb R^N)$.
By  assumptions,  $I_{\varepsilon_{n}}'( u_n) [ u_n] =0$ and $I_{\varepsilon_{n}}(u_{n})\rightarrow m_{V_{0}}^{\infty}$ write as
\begin{equation}\label{eq:ne}
\norm{u_n}_{\varepsilon_n}^2+\int_{\mathbb R^N}\phi_{\varepsilon_n, u_n}u_n^2=\int_{\mathbb R^N} f( u_n)u_n
\end{equation}
and
\[
\frac{1}{2}\norm{u_n}_{\varepsilon_n}^2+\frac{1}{4}\int_{\mathbb R^N}\phi_{\varepsilon_n, u_n}u_n^2-\int_{\mathbb R^N}F( u_n)=m_{ V_0}^{\infty}+o_n( 1)
\]
which combined together give
\[
\frac{1}{4}\int_{\mathbb R^N} f( u_n)u_n-\int_{\mathbb R^N}F( u_n)=\frac{1}{4}\left( \norm{u_n}_{\varepsilon_n}^2+\int_{\mathbb R^N}\phi_{\varepsilon_n, u_n}u_n^2\right)-\int_{\mathbb R^N}F( u_n) \leq m_{V_0}^{\infty}+o_n\left( 1\right).
\]
Using \eqref{f_{4}} we get
\[
0\leq \Big( \frac{1}{4}-\frac{1}{K}\Big)\int_{\mathbb R^N} f( u_n)u_n \leq m_{ V_0}^{\infty}+o_n( 1),
\]
and therefore, coming back to \eqref{eq:ne},  for some positive constant $C$ (independent on $n$)
\begin{equation}\label{acotacion}
\norm{u_n}_{H^{s}_{V_0}}\leq \norm{u_n}_{\varepsilon_n} \leq C.
\end{equation}

\medskip
We prove the following \medskip

\noindent {\bf Claim: } there exists $\left\{ \tilde{y}_n\right\} \subset \mathbb R^N$ and $R,c>0$ such that
$\liminf _{n\rightarrow \infty} \int_{B_R\left( \tilde{y}_n\right)}u_n^2 \geq c>0.$ \smallskip

\noindent Indeed, if it were not the case then
\[
\lim_{n\rightarrow \infty} \sup_{y\in \mathbb R^N} \int_{B_R\left( y\right)}u_n^2=0,\quad \text{for every $R>0$}.
\]
By Lemma \ref{gio5}, $u_n\rightarrow 0$ in $L^p( \mathbb R^N)$, for $2<p<2^*_{s}$
and then
\[
\int_{\mathbb R^N} f( u_n)u_n \rightarrow 0.
\]
Therefore $\norm{u_n}_{\varepsilon_n}^2+\int_{\mathbb R^N}\phi_{\varepsilon_n, u_n}u_n^2=o_n( 1)$, and also from
\[
0\leq \int_{\mathbb R^N} F( u_n) \leq \frac{1}{K}\int_{\mathbb R^N} f( u_n)u_n
\]
we have $\int_{\mathbb R^N} F( u_n)=o_n( 1)$. But then 
$\lim _{n\rightarrow \infty} I_{\varepsilon_n}( u_n) =m_{ V_0}^{\infty}=0$ which is a contradiction
and proves  our claim.

\medskip

Then the sequence $v_{n}:= u_{n}(\cdot + \tilde{y}_{n})$  is also bounded in $H^s( \mathbb R^N)$
and 
\begin{equation}\label{eq:convergenza}
v_n \rightharpoonup v\not\equiv0\quad \text{ in } \quad H^s( \mathbb R^N)
\end{equation}
since
\begin{equation*}\label{eq:nonulla}
\int_{B_R}v^2=\liminf_{n\rightarrow \infty} \int_{B_R} v_n^2=\liminf_{n\rightarrow \infty} \int_{B_R\left( \tilde{y}_n\right)} u_n^2\geq c>0,
\end{equation*}
by  the claim.

\medskip

Let now $t_{n}>0$ be such that $\tilde{v}_{n}:=t_{n}v_{n}\in {\cal{M}}_{V_{0}}$; the next step is to prove that
\begin{equation}\label{eq:EV0}
E_{V_{0}}(\tilde{v}_{n})\rightarrow m^{\infty}_{V_{0}}.
\end{equation} For this, note that
\begin{align*}
m_{ V_0}^{\infty} &\leq E_{V_0} ( \tilde{v}_n)=\frac{1}{2}\norm{\tilde{v}_n}_{V_0}^2-\int_{\mathbb R^N} F( u_n)\\
&=\frac{t_n^2}{2}\int_{\mathbb R^N} \Big[\abs{(-\Delta )^{s/2}u_n( x+\tilde{y}_n)}^2+V_0u_n^2( x+\tilde{y}_n) \Big]dx-\int_{\mathbb R^N} F( t_nu_n( x+\tilde{y}_n))dx\\
&=\frac{t_n^2}{2}\int_{\mathbb R^N} \abs{(-\Delta )^{s/2}u_n( z)}^2dz+\frac{t_n^2}{2}\int_{\mathbb R^N}V_0u_n^2(z)dz-\int_{\mathbb R^N} F( t_nu_n( z))dz\\
&\leq \frac{t_n^2}{2}\int_{\mathbb R^N} \abs{(-\Delta )^{s/2}u_n}^2+\frac{t_n^2}{2}\int_{\mathbb R^N}V( \varepsilon_nz)u_n^2+\frac{t_n^4}{4}\int_{\mathbb R^N}\phi_{\varepsilon_n, u_n}u_n^2-\int_{\mathbb R^N} F( t_nu_n)
\\ &= I_{\varepsilon_{n}}(t_{n}u_{n})
\end{align*}
and then
\[
m_{ V_0}^{\infty}\leq E_{V_{0}}(\tilde v_{n})\leq I_{\varepsilon_n}( t_nu_n)\leq I_{\varepsilon_n}( u_n)=m_{ V_0}^{\infty}+o_n( 1)
\]
which proves \eqref{eq:EV0}.

\medskip

We can prove now that $v_n\rightarrow v$ in $H^s( \mathbb R^N)$. 
As in the first part of the proof (where we proved the boundedness of $\{u_{n}\}$ in $H^{s}_{V_{0}}(\mathbb R^{N})$), it is easy to see that
\begin{equation*}\label{eq:bo}
 \left\{ \tilde{v}_n\right\} \subset {\cal M}_{V_0} \quad \text{and } \quad E_{V_{0}}(\tilde{v}_{n})\rightarrow m^{\infty}_{V_{0}}
 \  \Longrightarrow \ \norm{\tilde{v}_n}_{H^{s}_{V_{0}}}\leq C
\end{equation*}
and an analogous claim as before holds for the sequence $\{\tilde v_{n}\}$.
%
%
 Then $\tilde{v}_n \rightharpoonup \bar{v}$ in $H^s_{V_{0}}( \mathbb R^N)$ and 
 (as before)
there exists $\delta >0$ such that 
\begin{equation}\label{eq:nonvanish}
 0<\delta \leq \norm{\tilde v_n}_{H^{s}_{V_{0}}}.
\end{equation} This implies
\[
0<t_n\delta \leq \norm{t_nv_n}_{H^{s}_{V_{0}}}=\norm{\tilde{v}_n}_{H^{s}_{V_{0}}}\leq C, 
\]
showing that, up to subsequence,
$t_n\rightarrow t_0\geq 0$.  
If now $t_0=0$ using \eqref{acotacion} 
we derive
\[
0\leq \norm{\tilde{v}_n}_{H^{s}_{V_0}}=t_n \norm{v_n}_{H^{s}_{V_0}}\leq t_n C\rightarrow 0,
\]
so that $\tilde{v}_n\rightarrow 0$ in $H^{s}_{V_0}(\mathbb R^{N})$. From this and \eqref{eq:EV0} 
it follows $m_{V_{0}}^{\infty}=0$ which is absurd. So  $t_0>0$.
Then $t_nv_n \rightharpoonup t_0\bar v=:\tilde v$ in $H^s( \mathbb R^N)$
and by \eqref{eq:nonvanish} $\tilde{v}\not\equiv 0$.
%
%
By Lemma \ref{gio16} applied to $\{\tilde v_{n}\}$ we get $\tilde v_{n} \to \tilde v$ in $H^{s}(\mathbb R^{N})$
and then $v_{n} \to \bar v$. By \eqref{eq:convergenza} we deduce $v_{n}\to v$
and the first part of the proposition is proved.

\medskip

We proceed to prove the second part. We first state that $\left\{ y_n\right\}$ is bounded in $\mathbb R^N$
(here $y_{n}=\varepsilon_{n} {\tilde y_{n}}$ with $\tilde y_{n}$ given in the above claim).
Assume the contrary; then 

\medskip

1. if $V_\infty<\infty$, since  $\tilde{v}_n \rightarrow \tilde{v}$ in $H^s( \mathbb R^N)$ and $V_0<V_\infty$,
we have
\begin{align*}
m_{ V_0}^{\infty}&=\frac{1}{2}\norm{\tilde{v}}_{H^{s}_{V_0}}^2-\int_{\mathbb R^N} F( \tilde{v})<\frac{1}{2}\norm{\tilde{v}}_{H^{s}_{V_\infty}}^2-\int_{\mathbb R^N} F( \tilde{v})\\
&\leq \liminf_{n\rightarrow \infty}\frac{1}{2}\int_{\mathbb R^N} \abs{(-\Delta )^{s/2}\tilde{v}_n}^2+\lim_{n\rightarrow \infty}\left( \frac{1}{2}\int_{\mathbb R^N}V( \varepsilon_nx+y_n)\tilde{v}_n^2( x)dx-\int_{\mathbb R^N} F(  \tilde v_n )\right)\\
&=\liminf_{n\rightarrow \infty}\left( \frac{t_n^2}{2}\int_{\mathbb R^N} \abs{(-\Delta )^{s/2}u_n}^2+\frac{t_n^2}{2}\int_{\mathbb R^N}V( \varepsilon_nz)u_n^2-\int_{\mathbb R^N} F( t_nu_n) \right)\\
&\leq \liminf_{n\rightarrow \infty}\left( \frac{1}{2}\norm{t_nu_n}_{\varepsilon_n}^2-\int_{\mathbb R^N} F( t_nu_n) +\frac{t_n^4}{4}\int_{\mathbb R^N}\phi_{\varepsilon_n, u_n}u_n^2\right)
\end{align*}
from which
\begin{equation*}\label{eq:contrad}
m_{ V_0}^{\infty}< \liminf_{n\rightarrow \infty}I_{\varepsilon_n}( t_nu_n)\leq \liminf_{n\rightarrow \infty}I_{\varepsilon_n}( u_n)=m^{\infty}_{V_0}
\end{equation*}
which is a contradiction.  

\medskip

2. If $V_\infty=\infty$, we have
\begin{gather*}
\begin{split}
\int_{\mathbb R^N}V( \varepsilon_nx+y_n)v_n^2( x)dx \leq &\int_{\mathbb R^N} \abs{(-\Delta )^{s/2}v_n( x)}^2dx+\int_{\mathbb R^N}V( \varepsilon_nx+y_n)v_n^2( x)dx\\
&+\int_{\mathbb R^N}\phi_{\varepsilon_n, v_n}( x)v_n^2(x)dx
\end{split}
\\
=\int_{\mathbb R^N} f( v_n( x) )v_n( x)dx,
\end{gather*}
and  by the Fatou's Lemma we obtain the absurd
\[
\infty =\liminf_{n\rightarrow \infty} \int_{\mathbb R^N} f( v_n)v_n=\int_{\mathbb R^N} f( v)v.
\]

Then  $\{ y_n\}$ has to be bounded and we can assume $y_n \rightarrow y\in \mathbb R^N$. If $y\notin M$ then $V_0<V( y)$, and similarly to the computation made in case 1. above (simply replace $V_{\infty}$ with $V(y)$) we have a contradiction.
Hence $y\in M$ and the proof is thereby complete.
%
%
\end{proof}
For $\delta > 0$ (later it will be fixed conveniently)  
let $\eta$ be a smooth nonincreasing cut-off function defined in $[0,\infty)$ such that 
\begin{equation*}\label{eta}
\eta(s)=
\begin{cases}
1 & \mbox{ if } 0 \leq s \leq \delta/2\\
0 & \mbox{ if } s\geq \delta.
\end{cases}
\end{equation*}
Let $\mathfrak w_{V_{0}}$ be a ground state solution given in Lemma \ref{gio16} of problem \eqref{prob4} with $\mu=V_{0}$ and for any $y \in M$, let us define
\begin{eqnarray*}
\Psi_{\varepsilon , y }(x) := \eta (|\varepsilon x - y|)\mathfrak w_{V_{0}}\biggl(\frac{\varepsilon x -y}{\varepsilon}\biggl).
\end{eqnarray*}
Let $t_{\varepsilon}>0$ verifying $\max_{t\geq 0}I_{\varepsilon}(t\Psi_{\varepsilon,y})=I_{\varepsilon}(t_{\varepsilon}\Psi_{\varepsilon ,y}),$
so that $t_{\varepsilon}\Psi_{\varepsilon,y}\in \mathcal N_{\varepsilon}$, and let 
\begin{equation*}\label{Phi}
\Phi_{\varepsilon}: y\in M\mapsto  t_{\varepsilon}\Psi_{\varepsilon ,y}\in \cal{N}_{\varepsilon}.
\end{equation*}
By construction, $\Phi_{\varepsilon}(y)$ has compact support for any $y\in M$ and  it is easy to see that $\Phi_{\varepsilon}$ is a continuous map.

The next result will help us to define a map from $M$ to a suitable sublevel in the Nehari manifold.

\begin{lemma}\label{gio14}
The function $\Phi_{\varepsilon}$ satisfies
\begin{eqnarray*}
\lim_{\varepsilon \rightarrow 0^{+}}I_{\varepsilon}(\Phi_{\varepsilon}(y))=m^{\infty}_{V_{0}}, \ \text{uniformly} \ in \ y \in M.
\end{eqnarray*}
\end{lemma}
\begin{proof}
Suppose by contradiction that the lemma is false. Then there exist $\delta_{0}> 0$, $\{y_{n}\} \subset M$ and $\varepsilon_{n}\rightarrow 0^{+}$ such that
\begin{eqnarray}\label{1eq7}
| I_{\varepsilon_{n}}(\Phi_{\varepsilon_{n}}(y_{n}))- m^{\infty}_{V_{0}}|
\geq \delta _{0}.
\end{eqnarray}
Using Lebesgue's Theorem, we have
\begin{eqnarray}\label{eq:1aconv}
\lim_{n\rightarrow \infty} \norm{\Psi_{{\varepsilon_n},y_n}}_{\varepsilon_n}^2 &= &\norm{\mathfrak w_{V_0}}_{H^{s}_{V_0}}^2\,,\\
\lim_{n\rightarrow \infty} \int_{\mathbb R^N}F\left( \Psi_{\varepsilon_n,y_n}\right) &=& \int_{\mathbb R^N}F\left( \mathfrak w_{V_0} \right) \nonumber\,, \\
\lim_{n\rightarrow \infty} \norm{\Psi_{{\varepsilon_n},y_n}}_{H^{s}_{V_0}}^2 &= &
\norm{\mathfrak w_{V_0}}_{H^{s}_{V_0}}^2 \nonumber.
\end{eqnarray}

This last convergence implies that $\{\norm{\Psi_{{\varepsilon_n},y_n}}\}$ is bounded. From \eqref{terminoconvolucion}
\[
\int_{\mathbb R^N} \phi_{\varepsilon_n, \Psi_{\varepsilon_n,y_n}}\Psi_{\varepsilon_n,y_n}^2\leq \varepsilon_n^{\alpha -\theta}C_{\textrm e}\norm{\Psi_{\varepsilon_n,y_n}}^4,
\]
and then
\begin{equation}\label{eq:convazero}
\lim_{n\rightarrow \infty}\int_{\mathbb R^N} \phi_{\varepsilon_n, \Psi_{\varepsilon_n,y_n}}\Psi_{\varepsilon_n,y_n}^2=0.
 \end{equation}

Remembering that $t_{\varepsilon_n}\Psi_{\varepsilon_n,y}\in \mathcal N_{\varepsilon_n}$
(see few lines before the Lemma),
 the condition
$I'_{\varepsilon_n}( t_{\varepsilon_n}\Psi_{\varepsilon_n,y_n} ) [ t_{\varepsilon_n}\Psi_{\varepsilon_n,y_n} ]=0 $
means
\begin{equation}\label{terminoderivada}
\norm{\Psi_{{\varepsilon_n},y_n}}_{\varepsilon_n}^2+t_{\varepsilon_n}^2\int_{\mathbb R^N} \phi_{\varepsilon_n, \Psi_{\varepsilon_n,y_n}}\Psi_{\varepsilon_n,y_n}^2=\int_{\mathbb R^N}\frac{f\left(t_{\varepsilon_n}\Psi_{\varepsilon_n,y_n} \right)}{t_{\varepsilon_n}\Psi_{\varepsilon_n,y_n}}\Psi_{\varepsilon_n,y_n}^2.
\end{equation}
We now prove the following \medskip

\noindent {\bf Claim: }$\lim_{n\to +\infty}t_{\varepsilon_{n}} =1.$ \smallskip

 We begin by showing the boundedness of $\{t_{\varepsilon_{n}}\}$.
 Since $\varepsilon_n \rightarrow 0^{+}$,  we can assume $\delta/2<\delta/( 2\varepsilon_n)$ and then from \eqref{terminoderivada}, using \eqref{f_{5}} and making the change of variable $z:=( \varepsilon_nx -y_n)/\varepsilon_n$, we get
\begin{equation}\label{acotaciont_n}
\frac{\norm{\Psi_{{\varepsilon_n},y_n}}_{\varepsilon_n}^2}{t_{\varepsilon_n}^2}+\int_{\mathbb R^N} \phi_{\varepsilon_n, \Psi_{\varepsilon_n,y_n}}\Psi_{\varepsilon_n,y_n}^2 \geq \frac{f\left( t_{\varepsilon_n}\mathfrak w_{V_0}\left( \overline{z}\right)\right)}{\left( t_{\varepsilon_n}\mathfrak w_{V_0}\left( \overline{z}\right) \right)^3} \int_{B_{\delta/2}} \mathfrak w_{V_0}^4\left( z\right),
\end{equation}
where $\mathfrak w_{V_0}( \overline{z}):= \min_{{B}_{\delta/2}}\mathfrak w_{V_0}( z)$. 
If $\{t_{\varepsilon_{n}}\}$ were unbounded,
passing to the limit in $n$ in \eqref{acotaciont_n}, the left hand side would tend to $0$
(due to \eqref{eq:1aconv} and \eqref{eq:convazero}), the right hand side to $+\infty$
(due to \eqref{f_{4}}).
So we can assume that $t_{\varepsilon_n} \rightarrow t_0\geq 0$. 


For given $\xi >0$, by \eqref{terminof}, there exists $M_{\xi}>0$ such that
\begin{equation}\label{cerot0}
\int_{\mathbb R^N}\frac{f\left(t_{\varepsilon_n}\Psi_{\varepsilon_n,y_n} \right)}{t_{\varepsilon_n}\Psi_{\varepsilon_n,y_n}}\Psi_{\varepsilon_n,y_n}^2\leq \xi \int_{\mathbb R^N}\Psi_{\varepsilon_n,y_n}^2+M_\xi t_{\varepsilon_n}^{q-1}\int_{\mathbb R^N}\Psi_{\varepsilon_n,y_n}^{q+1}.
\end{equation}
Since  $\{\Psi_{\varepsilon_n,y_n} \}$ is bounded  in $H^s( \mathbb R^N)$, if $t_{0}=0$,
from \eqref{cerot0} we deduce
\[
\lim_{n\rightarrow \infty} \int_{\mathbb R^N}\frac{f\left(t_{\varepsilon_n}\Psi_{\varepsilon_n,y_n} \right)}{t_{\varepsilon_n}\Psi_{\varepsilon_n,y_n}}\Psi_{\varepsilon_n,y_n}^2=0,
\]
which joint  with \eqref{eq:convazero} and \eqref{terminoderivada} led to 
$
\lim_{n\rightarrow \infty} \norm{\Psi_{{\varepsilon_n},y_n}}_{\varepsilon_n}^2=0
$
 contradicting \eqref{eq:1aconv}.
Then $t_{\varepsilon_{n}}\to t_{0}>0$.
Now taking the limit in $n$ in \eqref{terminoderivada} we arrive at 
\[
\norm{\mathfrak w_{V_0}}_{H^{s}_{{V_0}}}^2=\int_{\mathbb R^N} \frac{f( t_0 \mathfrak w_{V_0})}{t_0}\mathfrak w_{V_0},
\]
and since $\mathfrak w_{V_0} \in {\cal M}_{V_0}$, it has to be $t_0=1$, which proves the claim. 

\medskip

Finally, note that
\begin{eqnarray*}
I_{\varepsilon_{n}}(\Phi_{\varepsilon_{n}}(y_{n}))&=&
\frac{t^{2}_{\varepsilon_{n}}}{2}\int_{\mathbb R^{N}}\abs{(-\Delta)^{s/2}\Psi_{\varepsilon_n,y_n}}^2 + \frac{t^{2}_{\varepsilon
_{n}}}{2}\int_{\mathbb R^{N}}V(\varepsilon_n x)\Psi_{\varepsilon_n,y_n}^2 \\
&+& \frac{t_{\varepsilon_n}^4}{4}\int_{\mathbb R^{N}}\phi_{\varepsilon_n, \Psi_{\varepsilon_n,y_n}}\Psi_{\varepsilon_n,y_n}^2 -\int_{\mathbb R^{N}}F\left(t_{\varepsilon_n} \Psi_{\varepsilon_n,y_n}\right).
\end{eqnarray*}
and then (by using the claim) $\lim_{n\rightarrow\infty}I_{\varepsilon_{n}}(\Phi_{\varepsilon_{n}}(y_{n}))=E_{V_{0}}(\mathfrak w_{V_{0}})=m^{\infty}_{{V_{0}}}$, which contradicts (\ref{1eq7}). Thus the Lemma holds.
\end{proof}
The remaining part of the paper mainly follows the arguments of \cite{gaetanogiovany}.

By Lemma \ref{gio14}, $h(\varepsilon):=|I_{\varepsilon}(\Phi_{\varepsilon}(y))-m^{\infty}_{V_{0}}|=o(1)$ for $\varepsilon\to 0^{+}$ uniformly in $y$, and then $I_{\varepsilon}(\Phi_{\varepsilon}(y))-m^{\infty}_{V_{0}}\leq h(\varepsilon)$. In particular the sublevel set in the Nehari manifold
\begin{equation*}\label{subnehari}
{\cal{N}}^{m^{\infty}_{V_{0}}+h(\varepsilon)}_{\varepsilon}:=\Big\{u\in{\cal{N}}_{\varepsilon}:I_{\varepsilon}(u)\leq m^{\infty}_{V_{0}}+h(\varepsilon)\Big\}
\end{equation*}
is not empty, since 
for sufficiently small $\varepsilon$,
\begin{equation}\label{phiepsilon}
\forall\, y\in M: \Phi_{\varepsilon}(y)\in {\cal{N}}^{m^{\infty}_{V_{0}}+h(\varepsilon)}_{\varepsilon}.
\end{equation}


From now on we fix a  $\delta >0$ in such a way that $M$ and $$M_{2\delta}:=\Big\{x\in \mathbb R^{N}: d(x,M)\leq 2\delta\Big\}$$ are homotopically equivalent ($d$ denotes the euclidean distance). Take a $\rho=\rho(\delta) > 0$  such that $M_{2\delta}\subset B_{\rho}$ and $\chi : \mathbb R^{N}\rightarrow \mathbb R^{N}$ be defined as follows
\begin{equation*}\label{chi}
\chi(x)= 
\begin{cases}
x & \mbox{ if } | x | \leq \rho\\
\rho \dis\frac{x}{|x|} & \mbox{ if } | x | \geq \rho.
\end{cases}
\end{equation*}
Define the {\sl barycenter map} $\beta_{\varepsilon}$ 
$$
\beta_{\varepsilon}(u):=\frac{\dis\int_{\mathbb R^{N}}\chi(\varepsilon x)
u^{2}(x)}{\dis\int_{\mathbb R^{N}}u^{2}(x)}\in \mathbb R^{N}
$$
for all $u\in W_{\varepsilon}$ with compact support.

We will take advantage of the  following results (see \cite[Lemma 8 and 9]{gaetanogiovany}).
\begin{lemma}\label{gio15} The function $\beta_{\varepsilon}$ satisfies
\begin{eqnarray*}
\lim_{\varepsilon \rightarrow 0^{+}}\beta_{\varepsilon}(\Phi_{\varepsilon}(y)) = y, \ \
\text{uniformly in }  y \in M.
\end{eqnarray*}
\end{lemma}
\begin{lemma}\label{gio18}
We have
\begin{eqnarray*}
\lim_{\varepsilon \rightarrow 0^{+}}\ \sup_{u \in\cal{N}^{m^{\infty}_{V_{0}}+h(\varepsilon)}_{\varepsilon}} \ \inf_{y \in M_{\delta}}\Big| \beta_{\varepsilon} (u)- y \Big|  = 0.
\end{eqnarray*}
\end{lemma}
In virtue of  Lemma \ref{gio18}, there exists $\varepsilon^{*}>0$ such that
$$\forall\,\varepsilon\in (0,\varepsilon^{*}]: \ \ \sup_{u\in {\mathcal N}^{m^{\infty}_{V_{0}}+h(\varepsilon)}_{\varepsilon}} d(\beta_{\varepsilon}(u), M_{\delta})<\delta/2.$$

Define now
$$M^{+}:=\Big\{x\in \mathbb R^{N}: d(x,M)\leq 3\delta/2\Big\}$$
so that  $M$ and $M^{+}$ are homotopically equivalent.

Now,  reducing  $\varepsilon^{*}>0$ if necessary,
we can assume that Lemma \ref{gio15},  Lemma \ref{gio18} and \eqref{phiepsilon} hold. Then by standard arguments the composed map
\begin{equation*}\label{fundamental}
M\stackrel{\Phi_{\varepsilon}}{\longrightarrow} {\cal{N}}^{m^{\infty}_{V_{0}}+h(\varepsilon)}_{\varepsilon}\stackrel{\beta_{\varepsilon}}{\longrightarrow}M^{+}
\quad \text{ is homotopic to the inclusion map.}
\end{equation*}

In case $V_{\infty}<\infty$, we eventually reduce $\varepsilon^{*}$ in such a way that also the Palais-Smale condition is satisfied in the interval  $(m^{\infty}_{V_{0}}, m^{\infty}_{V_{0}}+h(\varepsilon))$, see Proposition \ref{gio13}.
By well known properties of the category,
it is
\begin{eqnarray*}
\cat({\cal{N}}^{m^{\infty}_{V_{0}}+h(\varepsilon)}_{\varepsilon})\geq
\cat_{M^{+}}(M)
\end{eqnarray*}
and the Ljusternik-Schnirelman theory   ensures the existence of at least 
$\cat_{M^{+}}(M)=\cat(M)$ constraint critical points of $I_{\varepsilon}$
on ${\cal{N}}_{\varepsilon}$. The proof of the main Theorem \ref{th:main}
then follows by Proposition \ref{gio13}.

If $M$ is bounded and not contractible in itself, then the existence of another critical point
of $I_{\varepsilon}$ on $\mathcal N_{\varepsilon}$ follows from some ideas in \cite{BCP}.
We recall here the main steps for completeness.

The  goal is to  exhibit
a subset $\mathcal A\subset\mathcal N_{\varepsilon}$ such that
\begin{itemize}
\item[i)] $\mathcal A$ is not contractible in $\mathcal N_{\varepsilon}^{m^{\infty}_{V_{0}}+h(\varepsilon)}$, \medskip
\item[ii)] $\mathcal A$ is contractible in $\mathcal N_{\varepsilon}^{\bar c}=\{u\in \mathcal N_{\varepsilon}: I_{\varepsilon}(u)\leq \bar c\}$,  for some $\bar c>m^{\infty}_{V_{0}}+h(\varepsilon)$.
\end{itemize}
This would imply, since the Palais-Smale holds, that there is a critical level between
$m_{V_{0}}^{\infty}+h(\varepsilon)$ and $\bar c$.

\medskip

First note that when $M$ is not contractible and bounded
the compact set $\mathcal A:={\Phi_{\varepsilon}(M)}$ 
can not be contractible in ${\cal{N}}^{m^{\infty}_{V_{0}}+h(\varepsilon)}_{\varepsilon}$, proving i).

Let us denote, for $u\in W_{\varepsilon}\setminus\{0\}$,
with $t_{\varepsilon}(u)>0$ the unique positive number such that $t_{\varepsilon}(u) u\in \mathcal N_{\varepsilon}.$
Choose a function $u^{*}\in W_{\varepsilon}$ be such that $u^{*}\geq 0$,  $I_{\varepsilon}(t_{\varepsilon}(u^{*})u^{*})>m^{\infty}_{V_{0}}+h(\varepsilon)$
and consider the compact and contractible cone
$$\mathfrak C:=\Big\{tu^{*}+(1-t)u: t\in [0,1], u\in\mathcal A \Big\}.$$
Observe that, since the functions in  $\mathfrak C$
have to be positive on a set of nonzero measure, it is  $0\notin \mathfrak C$.
Now we project this cone on $\mathcal N_{\varepsilon}$: let
$$t_{\varepsilon}(\mathfrak C):=\Big\{t_{\varepsilon}(w)w: w\in \mathfrak C\Big\}\subset \mathcal N_{\varepsilon}$$
and set
$$\overline c:=\max_{t_{\varepsilon}(\mathfrak C)}I_{\varepsilon}>m^{\infty}_{V_{0}}+h(\varepsilon)$$
(indeed the maximum is achieved being $t_{\varepsilon}(\mathfrak C)$ compact).
Of course $\mathcal A\subset t_{\varepsilon}(\mathfrak C)\subset \mathcal N_{\varepsilon}$
and $t_{\varepsilon}(\mathfrak C)$ is contractible in $\mathcal N^{\bar c}_{\varepsilon}$:
we deduce ii).

Then there is a critical level for $I_{\varepsilon}$ greater than $m^{\infty}_{V_{0}} + h(\varepsilon)$,
hence different from the previous ones we have found. The proof of Theorem \ref{th:main} is complete.

\end{document}